\documentclass[11pt]{amsart}
\theoremstyle{plain}
\numberwithin{equation}{section}
\newtheorem{Theorem}{Theorem}
\newtheorem{Lemma}[Theorem]{Lemma}

\newtheorem{Proposition}[Theorem]{Proposition}

\newtheorem{Corollary}[Theorem]{Corollary}

\newtheorem{Remark}[Theorem]{Remark}

\newcommand{\be}{\begin{equation}}
\newcommand{\ee}{\end{equation}}
\newcommand{\phii}{\varphi}

\title{Characterization of potential smoothness and Riesz basis
property of Hill-Scr\"odinger operators with singular periodic potentials in terms of periodic,
antiperiodic and Neumann spectra}

\author{Ahmet Batal}
\address{Sabanci University, Orhanli,
34956 Tuzla, Istanbul, Turkey, E-mail:ahmetbatal@sabanciuniv.edu}

\begin{document}

\begin{abstract}
The Hill operators $Ly=-y''+v(x)y$, considered with singular complex valued
$\pi$-periodic potentials $v$ of the form $v=Q'$ with $Q\in L^2([0,\pi]),$ and subject to periodic, antiperiodic
or Neumann boundary conditions have discrete spectra. For
sufficiently large $n,$ the disc $\{z:;|z-n^2|<n\}$ contains two periodic (if $n$
is even) or antiperiodic (if $n$ is odd) eigenvalues $\lambda_n^-$,
$\lambda_n^+$ and one Neumann eigenvalue $\nu_n$. We show that rate of decay of the sequence $|\lambda_n^+-\lambda_n^-|+|\lambda_n^+-\nu_n|$ determines the potential smoothness, and there is a basis consisting of periodic (or antiperiodic) root functions if and only if for even (respectively, odd) $n$, $
\sup_{\lambda_n^+\neq
\lambda_n^-}\{|\lambda_n^+-\nu_n|/|\lambda_n^+-\lambda_n^-| \} <
\infty. $
\end{abstract}

\maketitle

\section{Introduction}
The Schr\"odinger operator
 \be
 \label{1}
  L y= -y''+v(x)y, \quad \quad \quad x \in \mathbb{R},
 \ee
considered with a real-valued $\pi$-periodic potential $v \in
L^2([0,\pi]),$ is self-adjoint and its spectrum has a band-gap
structure, i.e., it consists of intervals $[\lambda^+_{n-1},
\lambda^-_{n}]$ separated by spectral gaps (instability zones)
$(\lambda^-_{n}, \lambda^+_{n})$, $n \in \mathbb{N}$. The Floquet
theory (e.g., see \cite{E}) shows that the endpoints $\lambda^-_{n}$,
$\lambda^+_{n}$ of these gaps are eigenvalues of the same
differential operator (\ref{1}) but considered on the interval
$[0,\pi]$ with periodic boundary conditions for even $n$ and
antiperiodic boundary conditions for odd $n$.

Hochstadt \cite{Ho1, Ho2} discovered that there is a close relation
between the rate of decay of the {\em spectral gap}
$\gamma_n=\lambda^+_n-\lambda^-_n$ and the smoothness of the
potential $v$. He proved that every finite zone potential is a
$C^\infty $-function, and moreover, \emph{if v is infinitely
differentiable then $\gamma_n$ decays faster than any power of
$1/n.$} Later several authors \cite{LP}-\cite{MT} studied this
phenomenon and showed that \emph{if $\gamma_n$ decays faster than any
power of $1/n,$ then v is infinitely differentiable}. Moreover,
Trubowitz \cite{Tr} proved that \emph{$v$ is analytic if and only if
$\gamma_n$ decays exponentially fast}.

If $v$ is a complex-valued function then the operator (\ref{1}) is
not self-adjoint and we cannot talk about spectral gaps. But
$\lambda^\pm_n $ are well defined for sufficiently large $n$ as
eigenvalues of (\ref{1}) considered on the interval $[0,\pi]$ with
periodic or antiperiodic boundary conditions, so we set again
$\gamma_n = \lambda_n^+ - \lambda_n^-$ and call it spectral gap.
Again the potential smoothness determines the decay rate of
$\gamma_n,$ but in general the opposite is not true. The decay rate
of $\gamma_n$ has no control on the smoothness of a complex valued
potential $v$ by itself as the Gasymov paper \cite{Gas} shows.

Tkachenko \cite{Tk92}--\cite{Tk94} discovered that the smoothness of
complex potentials could be controlled if one consider, together
with the spectral gap $\gamma_n,$ the deviation
$\delta^{Dir}_n=\lambda^+_n-\mu_n$, where $\mu_n$ is the closest to
$n^2$ Dirichlet eigenvalue of $L.$ He characterized in these terms
the $C^{\infty}$-smoothness and analyticity of complex valued
potentials $v.$  Moreover, Sansuc and Tkachenko \cite{ST} showed
that $v$ is in the Sobolev space $H^a$, $a\in \mathbb{N}$ if and
only if $\gamma_n $ and $\delta^{Dir}_n$ are in the weighted
sequence space $\ell_a^2=\ell^2((1+n^2)^{a/2})$.

The above results have been obtained by using Inverse Spectral Theory.
Kappeler and Mityagin \cite{KaMi01} suggested another approach based
on Fourier Analysis. To formulate their results, let us recall that
the smoothness of functions could be characterized by weights $\Omega
= (\Omega (k)),$  and the corresponding weighted Sobolev spaces are
defined by
$$
H(\Omega) = \{v(x) = \sum_{k \in \mathbb{Z}} v_k e^{2ikx}, \quad
\sum_{k \in \mathbb{Z}} |v_k|^2 (\Omega(k))^2 < \infty \}.
$$
A weight $\Omega $ is called sub-multiplicative, if $ \Omega (-k) =
\Omega (k) $ and $ \Omega (k+m) \leq \Omega (k)\Omega (m) $ for $ k,m
\geq 0. $ In these terms the main result in \cite{KaMi01} says that
if $\Omega $ is a sub-multiplicative weight, then
 \be
  \label{i-ra}
 (A) \quad v\in H(\Omega) \quad \Longrightarrow \quad  (B) \quad
   (\gamma_n) ,  \left (\delta^{Dir}_n \right ) \in
  \ell^2(\Omega).
 \ee
Djakov and Mityagin \cite{DM3,DM5, DM15} proved the inverse
implication $(B) \Rightarrow (A)$  under some additional mild
restrictions on the weight $\Omega. $  Similar results were obtained
for 1D Dirac operators (see \cite{DM7,DM6,DM15}).

The analysis in \cite{KaMi01, DM3, DM5, DM15} is carried out under
the assumption $v\in L^2([0,\pi]).$  Using the quasi-derivative
approach of Savchuk-Shkalikov \cite{SS}, Djakov and Mityagin
\cite{DM16} developed a Fourier method for studying the spectra of
$L_{Per^\pm}$ and $L_{Dir}$ in the case of periodic singular
potentials and  extended the above results. They proved that if  $v
\in H^{-1}_{per} (\mathbb{R})$ and $\Omega$ is a weight of the form
$\Omega(m)=\omega(m)/|m|$ for  $m \neq 0,$ with $\omega$ being a
sub-multiplicative weight, then  $(A) \Rightarrow (B),  $ and
conversely, if in addition $(\log \omega (n))/n$ decreases to zero,
then $(B) \Rightarrow (A)$ (see Theorem 37 in \cite{DMH}).

A crucial step in proving the implications $(A) \Rightarrow (B)$ and
$(B) \Rightarrow (A)$ is the following statement (which comes from
Lyapunov-Schmidt projection method, e.g., see Lemma 21 in
\cite{DM15}): {\em For large enough $n,$ there exists a matrix $
\begin{pmatrix} \alpha_n(z) & \beta^+_n(z) \\ \beta^-_n(z) &
\alpha_n(z)\end{pmatrix}$ such that number $\lambda = n^2 +
z$ with $|z| <n/4 $ is a periodic or antiperiodic eigenvalue if and only
if $z$ is an eigenvalue of this matrix}. The entrees
$\alpha_n(z)=\alpha_n(z;v)$ and $\beta^{\pm}_n(z)=\beta^{\pm}_n(z;v)$
are given by explicit expressions in terms of the Fourier
coefficients of the potential $v$ and depend analytically on $z$ and
$v.$

The functionals $\beta^\pm_n $ give lower and  upper bounds for
the gaps and deviations (e.g., see Theorem 29 in  \cite{DMH}): {If
$v \in H^{-1}_{per} (\mathbb{R})$ then, for sufficiently large} $n,$
\begin{equation}
\label{A1} \frac{1}{72}(|\beta_n^+(z_n^*)|+|\beta_n^-(z_n^*)|) \leq
|\gamma_n|+|\delta^{Dir}_n| \leq 58
 (|\beta_n^+(z_n^*)|+|\beta_n^-(z_n^*)|),
\end{equation}
where $z_n^*=  \frac{1}{2}(\lambda^+_n + \lambda^-_n) - n^2.$ Thus,
the implications $(A) \Rightarrow (B)$ and $(B) \Rightarrow (A)$ are
equivalent, respectively, to
\begin{equation}
\label{A2} (\tilde{A}):\quad v\in H(\Omega) \quad \Longrightarrow
\quad (\tilde{B}): \quad (|\beta_n^+(z_n^*)|+|\beta_n^-(z_n^*)|)\in
\ell^2 (\Omega),
\end{equation}
and $(\tilde{B}) \Rightarrow (\tilde{A}).$ In this way the problem of
analyzing the relationship between potential smoothness and decay
rate of the sequence $(|\gamma_n| + |\delta_n^{Dir}|)$ is reduced to
analysis of the functionals $\beta^{\pm}_n (z).$

The asymptotic behavior of $\beta^{\pm}_n (z)$ (or $\gamma_n$ and
$\delta^{Dir}_n$) plays also a crucial role in studying the Riesz
basis property of the system of root functions of the operators
$L_{Per^\pm}.$ In \cite[Section 5.2]{DM15}, it is shown (for
potentials $v \in L^2 ([0,\pi])$) that if the ratio $\beta^+_n
(z_n^*)/\beta^-_n (z_n^*) $ is not separated from $0$ or $\infty$
then the system of root functions of $L_{Per^\pm}$ does not contain a
Riesz basis (see Theorem 71 and its proof therein). Theorem 1 in
\cite{DM25} (or Theorem 2 in \cite{DM25a}) gives, for wide classes of
$L^2$-potentials,  a criterion for Riesz basis property in the same
terms. In its most general form, for singular potentials, this
criterion reads as follows (see Theorem 19 in \cite{DM28}):

{\em Criterion 1. Suppose $v\in H^{-1}_{per}(\mathbb{R});$ then the
set of root functions of $L_{Per^\pm}(v)$ contains Riesz bases if
and only if
\begin{equation}
\label{DM1} 0 < \inf_{\gamma_n\neq 0} |\beta^-_n (z_n^*)|/|\beta^+_n
(z_n^*)|, \quad \sup_{\gamma_n\neq 0} |\beta^-_n (z_n^*)|/|\beta^+_n
(z_n^*)|< \infty,
\end{equation}
where $n$ is even (respectively odd) in the case of periodic
(antiperiodic) boundary conditions.}

In \cite{GT11}  Gesztesy and Tkachenko obtained the following result.

{\em Criterion 2.   If $v \in L^2 ([0,\pi]),$ then there is a Riesz
basis consisting of root functions of the operator $L_{Per^\pm}(v)$
if and only if
\begin{equation}
\label{GT1} \sup_{\gamma_n\neq 0}|\delta^{Dir}_n|/|\gamma_n|<\infty,
\end{equation}
where $n$ is even (respectively odd) in the case of periodic
(antiperiodic) boundary conditions.}

They also noted that a similar criterion holds if (\ref{GT1}) is
replaced by
\begin{equation}
\label{GT2} \sup_{\gamma_n\neq 0}|\delta^{Neu}_n|/|\gamma_n|<\infty,
\end{equation}
where $\delta^{Neu}_n=\lambda^+_n - \nu_n $ and $\nu_n $ is the $n$th
Neumann eigenvalue.

Djakov and Mityagin \cite[Theorem 24]{DM28} proved, for potentials
$v\in H^{-1}_{per}(\mathbb{R}),$ that the conditions ({\ref{DM1}) and
(\ref{GT1}) are equivalent, so (\ref{GT1}) gives necessary and
sufficient conditions for Riesz basis property for singular
potentials as well.

On the other hand, recently the author has shown \cite{AB01} for potentials $v\in L^p([0,\pi]), $ $ p>1,$
that the Neumann version of Criterion 2 holds, and the potential
smoothness could be characterized by the decay rate of $ |\gamma_n|
+ |\delta_n^{Neu}|.$ However, whether the same is true for potentials $v\in H^{-1}_{per}(\mathbb{R}),$ even for $v\in L^1([0,\pi])$, was still unknown. In this paper we show that the answer is affirmative. More precisely, Theorem \ref{ThmNeu} and Theorem \ref{Theorem} hold. However before stating these theorems we want to explain what was the difficulty which prevent us from stating these theorems for potentials worse than $L^p([0,\pi])$ in \cite{AB01} and how we overcome this difficulty in this paper.

The main inequality we used to prove our claims in \cite{AB01} is the inequality (3.16) which states that for all $bc$ $D(P_n-P_n^0)$ are uniformly bounded as a sequence of operators from $L^2([0,\pi])$ to $L^\infty([0,\pi])$, i.e.,
\be
\label{1-M} \|D(P_n-P_n^0)\|_{L^2([0,\pi])\rightarrow L^\infty([0,\pi])} \leq M,
\ee
where $M$ is an absolute constant, $D=\frac{d}{dx}$ and $P_n$ and $P_n^0$ are the Cauchy Riesz projections into the two dimensional invariant subspaces of $L_{Per\pm}$ and $L^0_{Per\pm}$, respectively (see (\ref{PP0})).
However in the case of $v\in L^1([0,\pi])$, the author is not even sure of that $D(P_n-P_n^0)$ is a bounded operator from $L^2([0,\pi])$ to $L^\infty([0,\pi])$, so such an inequality as (\ref{1-M}) may not even exist if $v\in L^1([0,\pi])$. Hence the results of \cite{AB01} cannot be generalized further for the potentials $v\in L^1([0,\pi])$ using the same method.

On the other hand (\ref{1-M}) is used only for its corollary (see (4.17) and (4.32) in \cite{AB01}) which states that $|G_n'(0)-{G_n^0}'(0)|$ are uniformly bounded where $G_n$ is a unit vector in $Ran P_n$ and $G_n^0=P_n^0 G_n$. So actually we do not need the uniform boundedness of $G_n'-{G_n^0}'$ for all $x$ values but only for $x=0$ and not for all $bc$'s but only for $bc= Per_\pm$. Moreover it is also easy to see that even uniform boundedness is too strong and what we actually need is $|G_n'(0)-{G_n^0}'(0)|/n$  to converge to zero. Hence, in the case of potentials worse than $L^p([0,\pi])$, even if we cannot claim such an inequality as (\ref{1-M}), we may still hope to find a good estimate for the difference $G_n'(0)-{G_n^0}'(0)$. Actually one should also replace the usual derivative of $G_n$ by its quasi derivative $G_n^{[1]}$ in the case of singular potentials $v\in H^{-1}_{per}(\mathbb{R})$. This is what we do in the present paper. The main difference between \cite{AB01} and this one is that in the present paper, we do not try to find such an estimate as (\ref{1-M}) but we directly prove in Proposition \ref{mainprop} that $|G_n^{[1]}(0)-{G_n^0}'(0)|/n$ converges to zero. 

Now we state our main theorems.
\begin{Theorem}
\label{ThmNeu}  Suppose $v\in H^{-1}_{per}(\mathbb{R})$ and $\Omega$  is
a weight of the form $\Omega(m)=\omega(m)/m$ for  $m \neq 0,$ where
$\omega$ is a sub-multiplicative weight. Then
\begin{equation}
\label{B10} v \in H(\Omega) \Longrightarrow (|\gamma_n|),
(|\delta^{Neu}_n|) \in \ell^2 (\Omega);
\end{equation}
conversely, if in addition $(\log \omega (n))/n$ decreases to zero,
then
\begin{equation}
\label{B20}  (|\gamma_n|), (|\delta^{Neu}_n|) \in \ell^2
(\Omega)\Longrightarrow v \in H(\Omega).
\end{equation}
If $ \lim \frac{\log \omega (n)}{n} >0,$ (i.e. $\omega $ is of
exponential type), then
\begin{equation}
\label{B30} (\gamma_n), (\delta^{Neu}_n) \in \ell^2 (\Omega) \;
\Rightarrow \; \exists  \varepsilon >0: \; v \in H(e^{\varepsilon
|n|} ).
\end{equation}
\end{Theorem}

 \begin{Theorem}
 \label{Theorem}
  If $v \in H^{-1}_{per}(\mathbb{R}),$ then there is a Riesz
basis consisting of root functions of the operator $L_{Per^\pm}(v)$
if and only if
\begin{equation}
\label{GT10} \sup_{\gamma_n\neq 0}|\delta^{Neu}_n|/|\gamma_n|<\infty,
\end{equation}
where $n$ is respectively even (odd) for periodic (antiperiodic)
boundary conditions.
 \end{Theorem}

We do not prove Theorem~\ref{ThmNeu} and Theorem~\ref{Theorem} directly, but show
that they are valid by reducing their proofs to Theorem~37 in
\cite{DMH} and Theorem~19 in \cite{DM28}, respectively. For this end
we prove the following theorem which generalizes Theorem 3 in \cite{AB01}.

 \begin{Theorem}
 \label{Asymp}
If $v \in H^{-1}_{per}(\mathbb{R}),$ then, for sufficiently large $n,$
\begin{equation}
\label{A10} \frac{1}{80} (|\beta_n^+(z_n^*)|+|\beta_n^-(z_n^*)|)
\leq |\gamma_n|+|\delta^{Neu}_n| \leq 19
 (|\beta_n^+(z_n^*)|+|\beta_n^-(z_n^*)|).
\end{equation}
\end{Theorem}

Next we show that Theorem \ref{Asymp} implies Theorem~\ref{ThmNeu}
and Theorem~\ref{Theorem}}. By Theorem~29 in \cite{DMH} and
Theorem~\ref{Asymp}, (\ref{A1}) and (\ref{A10}) hold simultaneously,
so the sequences $(|\gamma_n|+|\delta^{Dir}_n|)$ and
$(|\gamma_n|+|\delta^{Neu}_n|)$ are asymptotically equivalent.
Therefore, every claim in Theorem~\ref{ThmNeu} follows from the
corresponding assertion in
 \cite[Theorem 37]{DMH}.

On the other hand the asymptotic equivalence of
$|\gamma_n|+|\delta^{Dir}_n|$  and $|\gamma_n|+|\delta^{Neu}_n|$
implies that $\sup_{\gamma_n\neq 0}|\delta^{Dir}_n|/|\gamma_n|<\infty$ if and only if
$\sup_{\gamma_n\neq 0}|\delta^{Neu}_n|/|\gamma_n|<\infty,$ so (\ref{GT1}) and (\ref{GT10})
hold simultaneously if $v\in H^{-1}_{per} (\mathbb{R}).$ By Theorem 24 in
\cite{DM28}, (\ref{GT1}) gives necessary and sufficient conditions
for the Riesz basis property if $v\in H^{-1}_{per} (\mathbb{R}).$
Hence, Theorem~\ref{Theorem} is proved.

\section{Preliminary Results}
\vspace{3mm}
We consider the Hill-Schr\"odinger operator on the interval $[0,\pi]$ generated by the differential
expression
\be
\label{ilk}
\ell(y)=-y''+v\cdot y
\ee
where $v$ is in the space of $\pi-$periodic distributions $H_{per}^{-1}(\mathbb{R}) \subset H_{loc}^{-1}(\mathbb{R})$. We define the appropriate boundary conditions and corresponding domains of the operator following the approach suggested and developed by A. Savchuk and A. Shkalikov \cite{SS01,SS}  and R. Hryniv and Ya. Mykytyuk \cite{HM01}. It is known that (see  \cite{HM01}, Remark 2.3) each $v\in H_{per}^{-1}(\mathbb{R})$ has the form
\be
\label{vcq}
v=C+Q'
\ee
for some constant $C$ and an almost everywhere $\pi-$periodic $Q\in L^2([0,\pi])$. From now on we assume that $C=0$ since a constant shift of the operator results in a shift of the spectra but the objects we analyze i.e., root functions, spectral gaps and deviations, do not change. In view of (\ref{vcq}), the differential expression
(\ref{ilk}) can be written as
\be
\label{iki}
\ell(y)=-(y'-Qy)'-Qy'.
\ee
The expression $y'-Qy$ is called the quasi-derivative of $y$. For each of the following boundary conditions ($bc$)\footnote{Note that, for a given potential $v$, $Q$ is determined up to a constant shift, i.e., $Q$ can be replaced by $Q+z$ for any constant $z$. This freedom of choice of $Q$ has no effect on how the operator acts, neither on the periodic, anti-periodic or Dirichlet $bc$'s but it does change the Neumann $bc$ we consider. So the above definition of Neumann $bc$ describes a family of $bc$'s which depends on the choice of $Q$. In particular, if $v\in L^1([0,\pi])$, then $Q$ is absolutely continuous and Neumann $bc$ we defined above can be rewritten as $y'(0)=t y(0)$ and $y'(\pi)=t y(\pi)$, where the parameter $t=Q(0)=Q(\pi)$ can be any complex number since we are free to shift $Q$. Hence any result we obtain about the Neumann $bc$ as defined above applies to all members of this family of $bc$'s in the case of $v\in L^1([0,\pi])$ including the usual Neumann $bc$ where $t=0$.}
 \begin{align}
&\text{Periodic}\;(bc=Per^+): \quad\quad\quad y(0) = y (\pi), \;\;\;\; (y^\prime-Qy)(0) = (y'-Qy) (\pi);\nonumber \\
&\text{Antiperiodic}\;(bc=Per^-): \quad y(0) = -y (\pi), \;\; (y^\prime-Qy)(0) = -(y'-Qy) (\pi);\nonumber \\
&\text{Dirichlet}\;(bc=Dir): \quad\quad\quad\; y(0) = y (\pi)=0; \nonumber\\
&\text{Neumann}\;(bc=Neu): \quad\quad\;\, (y^\prime-Qy)(0) = (y'-Qy) (\pi)=0; \nonumber
\end{align}
we consider the closed operator $L_{bc}$, acting as $L_{bc}\,y=\ell(y)$ in the domain
\begin{align}
Dom(L_{bc})= \{y\in W^1_2([0,\pi])&\; :\; y'-Qy \in W^1_1([0,\pi]),\nonumber\\
&\ell(y)\in L^2([0,\pi]),\; \text{and}\;y \; \text{satisfies}\; bc\}.\nonumber
\end{align}
 For each $bc$, $Dom(L_{bc})$ is dense
in $L^2([0,\pi])$ and  $L_{bc}=L_{bc}(v)$ satisfies
 \be
 \label{L-L}
 (L_{bc}(v))^*=L_{bc}(\overline{v})\;\;\text{for};\;bc=Per^\pm,\;Dir,\;Neu,
 \ee
where $(L_{bc}(v))^*$ is the adjoint operator and  $\overline{v}$ is the
conjugate of $v$, i.e., $\langle \overline{v}, h \rangle = \overline{\langle v,\overline{h} \rangle}$ for all test functions $h$. In the classical case where $v\in L^2([0,\pi])$, (\ref{L-L}) is a well known
fact. In the case where $v\in H^{-1}_{per}(\mathbb{R})$ it is
explicitly stated and proved for $bc=Per^\pm,\;Dir$ in \cite{DM16},
see Theorem~6 and Theorem~13 there. Following the same argument as in
\cite{DM16} one can easily see that it holds for $bc=Neu$ as well.

If $v=0$ we write $L^0_{bc},$ (or simply $L^0$). The spectra and
eigenfunctions of $L^0_{bc}$ are as follows:

($\tilde{a}$) $ Sp (L^0_{Per^+}) = \{n^2, \; n = 0,2,4, \ldots \};$
its eigenspaces are  $\mathcal{E}^0_n = Span \{e^{\pm inx} \} $ for
$n>0 $ and $E^0_0 = \{ const\}, \; \; \dim \mathcal{E}^0_n = 2 $ for
$n>0, $ and $\dim \mathcal{E}^0_0 = 1. $

($\tilde{b}$) $ Sp (L^0_{Per^-}) = \{n^2, \; n = 1,3,5, \ldots \};$
its eigenspaces are  $\mathcal{E}^0_n = Span \{e^{\pm inx} \}, $ and
$ \dim \mathcal{E}^0_n = 2. $

($\tilde{c}$) $ Sp (L^0_{Dir}) = \{n^2, \; n \in \mathbb{N} \};$ each
eigenvalue $n^2 $ is simple; its eigenspaces are $\mathcal{S}^0_n =
Span \{ s_n (x)\}, $ where $s_n (x)$ is the corresponding normalized
eigenfunction $ s_n (x) = \sqrt{2} \sin nx.$

($\tilde{d}$) $ Sp (L^0_{Neu}) = \{n^2, \; n \in \{0\}\cup\mathbb{N}
\};$ each eigenvalue $n^2 $ is simple; its eigenspaces are
$\mathcal{C}^0_n = Span \{ c_n (x)\}, $ where $c_n (x)$ is the
corresponding normalized eigenfunction $ c_0 (x) = 1$, $c_n (x) =
\sqrt{2} \cos nx $ for $n>0.$

 The sets of indices $2\mathbb{Z}$,
 $2\mathbb{Z}+1$,  $\mathbb{N}$, and $\{0\}\cup\mathbb{N}$ will be denoted by
 $\Gamma_{Per^+}$, $\Gamma_{Per^-}$, $\Gamma_{Dir}$ and
 $\Gamma_{Neu}$, respectively. For each $bc$,
 we consider the corresponding canonical orthonormal basis of $L^0_{bc}$, namely
 $\mathcal{B}_{Per^+}=\{e^{inx}\}_{n\in\Gamma_{Per^+}}$,
  $\mathcal{B}_{Per^-}=\{e^{inx}\}_{n\in\Gamma_{Per^-}}$,
 $\mathcal{B}_{Dir}=\{s_n(x)\}_{n\in\Gamma_{Dir}}$, $\mathcal{B}_{Neu}=\{c_n(x)\}_{n\in\Gamma_{Neu}}.$

In \cite{DM16}, Djakov and Mityagin developed a Fourier method for studying the operators $L_{bc}$, $bc=Per^\pm$, $Dir$, with $ H_{per}^{-1}(\mathbb{R})$ potentials. To summarize their results let us denote the Fourier coefficients of a function $f\in L^1([0,\pi])$ with respect to the basis $\mathcal{B}_{bc}$ by $\widehat{f}^{bc}_k$, i.e.,
\be
\widehat{f}^{bc}_k=\frac{1}{\pi}\int_0^\pi f(x)\overline{u^{bc}_k(x)}dx,\quad \quad k\in \Gamma_{bc}\quad u^{bc}_k(x)\in \mathcal{B}_{bc}.
 \ee
Set also $V_+(k)=ik\widehat{Q}^{Per^+}_k$, $\widetilde{V}(0)=0$, and $\widetilde{V}(k)=k\widehat{Q}^{Dir}_k$. Let $\ell^2_1(\Gamma_{bc})$ be the weighted sequence space with weight $(1+k^2)^{1/2}$, i.e.,
\be
\ell^2_1(\Gamma_{bc})= \{a=(a_k)_{k\in \Gamma_{bc}}\; :\; \sum_{k\in \Gamma_{bc}}(1+k^2)|a_k|^2 < \infty\}.
\ee

Consider the unbounded operators $\mathcal{L}_{bc}$ acting in $\ell^2(\Gamma_{bc})$ as $\mathcal{L}_{bc}\,a=b=(b_k)_{k\in \Gamma_{bc}}$, where
\be
\label{dikkat}
b_k=k^2 a_k+ \sum_{m\in\Gamma_{bc}}V_+(k-m)a_m
\quad\;\text{for}\;\, bc=Per^\pm,
\ee
\be
\label{dikkat2}
b_k=k^2 a_k+ \frac{1}{\sqrt{2}}\sum_{m\in\Gamma_{Dir}}\left(\widetilde{V}(|k-m|)-\widetilde{V}(k+m)\right)a_m\quad\;\text{for}\;\, bc=Dir,
\ee
respectively in the domains
\be
Dom({\mathcal{L}_{bc}})=\{a\in\ell^2_1(\Gamma_{bc})\; :\;  \mathcal{L}_{bc}\,a \in \ell^2(\Gamma_{bc})\}.
\ee
Then for $bc=Per^\pm,\,Dir$ we have (Theorem 11 and 16 in \cite{DM16} )
\be
\label{FLF}
Dom(L_{bc})=\mathcal{F}_{bc}^{-1}(Dom(\mathcal{L}_{bc}))\quad and \quad
L_{bc}=\mathcal{F}_{bc}^{-1}\circ\mathcal{L}_{bc}\circ \mathcal{F}_{bc},
\ee
where $\mathcal{F}_{bc}: L^2([0,\pi])\rightarrow\ell^2(\Gamma_{bc})$ is defined by
\be
\mathcal{F}_{bc}(f)=(\widehat{f}^{bc}_k)_{k\in\Gamma_{bc}}.
\ee
Similar facts hold in the case of Neumann $bc$ as well. Indeed let us construct the unbounded operator $\mathcal{L}_{Neu}$ acting as $\mathcal{L}_{Neu}\,a=b$, where
\be
\label{dikkat3}
b_k=k^2 a_k+ \widetilde{V}(k)a_0+ \frac{1}{\sqrt{2}}\sum_{m=1}^\infty \left(\widetilde{V}(|k-m|)+\widetilde{V}(k+m)\right)a_m,
\ee
in the domain
\be
Dom({\mathcal{L}_{Neu}})=\{a\in\ell^2_1(\Gamma_{Neu})\; :\;  \mathcal{L}_{Neu}\,a \in \ell^2(\Gamma_{Neu})\}.
\ee
The following proposition implies that (\ref{FLF}) holds in the case of Neumann $bc$ as well.
\begin{Proposition}
\label{pneu}
In the above notations,
\be
y\in Dom(L_{Neu})\quad\text{and}\quad L_{Neu}\,y=h
\ee
if and only if
\be
\widehat{y}=(\widehat{y}^{Neu}_k)_{k\in \Gamma_{Neu}} \in Dom({\mathcal{L}_{Neu}})\quad\text{and}\quad
\mathcal{L}_{Neu}\,\widehat{y}=\widehat{h},
\ee
where $\widehat{h}=(\widehat{h}^{Neu}_k)_{k\in \Gamma_{Neu}}$.
\end{Proposition}
We omit the proof of Proposition \ref{pneu} because it is very similar to the proof of Proposition 15 in \cite{DM16}.

In view of (\ref{FLF}), from now on, for all $bc$, we identify $L_{bc}$ acting on the function space $L^2([0,\pi])$ with $\mathcal{L}_{bc}$ which acts on the corresponding sequence space $\ell^2(\Gamma_{bc})$ and use one and the same notation $L_{bc}$ for both of them. Moreover, the matrix elements of an operator $A$ acting on the sequence space $\ell^2(\Gamma_{bc})$ will be denoted by $A^{bc}_{nm}$, where $n,m \in \Gamma_{bc}$. The norm of a function $f\in L^a([0,\pi])$ and an operator $A$ from $L^a([0,\pi])$ to $L^b([0,\pi])$ for $a,b \in [1,\infty]$ will be denoted by $\|f\|_{{}_{{}_a}}$ and $\|A\|_{{}_{{}_{a\rightarrow b}}}$, respectively. We may also write $\|f\|$ and $\|A\|$ instead of $\|f\|_{{}_{{}_2}}$ and $\|A\|_{{}_{{}_{2\rightarrow 2}}}$, respectively.

By (\ref{dikkat}), (\ref{dikkat2}), and (\ref{dikkat3}) we see that $L_{bc}$ has the form
\be
\label{L0VC}
L_{bc}=L^0+V,
\ee
where we define $L^0$ and $V$, acting on the corresponding sequence space $\ell^2(\Gamma_{bc})$, by their matrix representations
\begin{align}
\label{L0}
L^0_{km}=&\,k^2\delta_{km}\hspace{64mm}\text{for all}\;bc,
\\
V_{km}=&\,V_+(k-m)\hspace{56mm}\text{for}\,\;bc=Per^\pm,
\\
V_{km}=&\,\frac{1}{\sqrt{2}}\left(\widetilde{V}(|k-m|)-\widetilde{V}(k+m)\right)\hspace{23mm}\text{for}\,\;bc=Dir,
&
\\
\label{Vneu}
V_{km}=&\,A_m\left(\widetilde{V}(|k-m|)+\widetilde{V}(k+m)\right)
\hspace{23mm}\text{for}\,\;bc=Neu.
\end{align}
Here $\delta_{km}$ are the Kronecker symbols and
\begin{equation}
\label{Am} A_m=
\begin{cases}1/\sqrt{2} & \text{if} \quad m \neq 0, \\
1/2 & \text{if} \quad m=0. \\
\end{cases}
\end{equation}
Note that in the notations of $L^0$ and $V$ the dependence on the boundary conditions is suppressed.

We use the perturbation formula (see \cite{DM16}, equation (5.13))

 \be
 \label{2-res}
  R_{\lambda}=R^0_{\lambda}+\sum_{s=1}^{\infty}K_{\lambda}(K_{\lambda}V
  K_{\lambda})^s K_{\lambda},
  \ee
  where $R_{\lambda}=(\lambda-L_{bc})^{-1}$, $R^0_{\lambda}=(\lambda-L^0_{bc})^{-1}$ and $K_{\lambda}$ is
  a square root of $R^0_{\lambda}$, i.e.,  $K_{\lambda}^2=
  R^0_{\lambda}$. Of course, (\ref{2-res}) makes sense only if the
  series on the right converges.

  By (\ref{L0}) we see that the matrix representation of $R^0_{\lambda}$ is
  \be
    (R^0_{\lambda})^{bc}_{km}=\frac{1}{\lambda-m^2}\delta_{km}.
    \quad k,m\in \Gamma_{bc}
    \ee
We define a square root $K=K_{\lambda}$ of $R^0_{\lambda}$ by
choosing its matrix representation as
    \be
    \label{Klamda}
    (K_{\lambda})^{bc}_{km}=\frac{1}{(\lambda-m^2)^{1/2}}\delta_{km}, \quad k,m\in
    \Gamma_{bc},
    \ee
where $z^{1/2}=|z|^{1/2}e^{i\theta/2}$ for $z=|z|e^{i\theta}$,
$\theta \in [0,2\pi)$. If
 \be
 \label{2-kvk}
 \|K_{\lambda}V K_{\lambda}\|_{2\rightarrow 2} < 1,
 \ee
then $R_{\lambda}$ exists.

Let
 \begin{align}
 &H^N=\{\lambda \in \mathbb{C} : Re \; \lambda  \leq N^2+N\},\\
 &R_N=\{\lambda \in \mathbb{C} : -N\leq Re \;\lambda\leq N^2+N,\quad |Im \lambda|<N\},\\
 &H_n=\{\lambda \in \mathbb{C} : (n-1)^2\leq Re \;\lambda \leq (n+1)^2\},\\
 &G_n=\{\lambda \in \mathbb{C} : n^2-n\leq Re \;\lambda \leq n^2+n\},\\
 \label{ere} &D_n=\{ \lambda \in \mathbb{C} : |\lambda -n^2|< r_n\},
\end{align}
where $r_n=n\tilde{\varepsilon}_n$ for some sequence $\tilde{\varepsilon}_n$ decreasing to zero.
Assuming only $v\in H_{per}^{-1}(\mathbb{R})$, Djakov and Mityagin showed (see
\cite{DM16}, Lemmas 19 and 20) that there exists $N>0$, $N\in
\Gamma_{bc}$ such that (\ref{2-kvk}) holds for $\lambda \in H^N
\backslash R_N$ and also for all $n>N$, $n\in \Gamma_{bc}$
(\ref{2-kvk}) holds for $\lambda \in H_n \backslash D_n$ if
$bc=Per^{\pm}$ and for $\lambda \in G_n \backslash D_n$ if $bc=Dir$.
Therefore, the following localization of the spectra holds:
  \be
  \label{local}
   Sp(L_{bc}) \subset R_N \cup \bigcup_{n>N, n\in
\Gamma_{bc}}D_n, \;\;\; bc=Per^\pm,\;Dir.
 \ee
Moreover, using the method of continuous parametrization of the
potential $v$, they showed that spectrum is discrete and
$$
 \sharp(Sp(L_{Per^+})\cap R_N)=2N+1, \quad
 \sharp(Sp(L_{Per^+})\cap D_n)=2,\quad n>N, n\in
 \Gamma_{Per^+},
$$$$
 \sharp(Sp(L_{Per^-})\cap R_N)=2N, \quad
 \sharp(Sp(L_{Per^-})\cap D_n)=2,\quad n>N, n\in
 \Gamma_{Per^-} ,
$$$$
 \sharp(Sp(L_{Dir})\cap R_N)=N, \quad  \sharp(Sp(L_{Dir})\cap D_n)=1,\quad n>N, n\in
 \Gamma_{Dir}.
$$
\begin{Remark}
\label{remark}
Although in \cite{DM16} Djakov and Mityagin formulated these lemmas for the discs $D_n$ with fixed radius $n$ they also pointed out (see the remark after Theorem 21) that the disks $D_n$
can be chosen as we defined in (\ref{ere}). Hence the localization of the spectra
can be sharpen for all $bc$'s we consider.
\end{Remark}
For Neumann $bc$ the situation is similar. The Neumann
eigenfunctions $c_k(x)$ of the free operator are uniformly bounded
and form an orthonormal basis, so using the same argument as in \cite{DM16}
one can similarly localize the spectrum $Sp(L_{Neu})$ after showing
that (\ref{2-kvk}) holds for $ \lambda \not \in R_N \cup \left
\{\bigcup_{n>N, n\in \Gamma_{Neu}}D_n \right \}.$ To be more specific first note that
Hilbert-Schmidt norm
\be
\label{HS}
\|A\|_{HS}=\bigg(\sum_{k,m}|A_{km}|^2\bigg)^{1/2}
\ee
of an operator $A$ majorizes its $L^2$ norm $\|A\|$. In \cite{DM16} (inequality (5.22)) it is shown that
\be
\label{ineq}
\|(K_{\lambda}V K_{\lambda})^{Dir}\|^2_{HS}\leq \sum_{k,m\in \mathbb{Z}}\frac{(k-m)^2|\widehat{Q}^{Dir}_{|k-m|}|^2}{|\lambda-k^2||\lambda-m^2|},
\ee
($\widehat{Q}^{Dir}_0$ is defined to be zero for convenience). Then, using this estimate, it was shown that Lemma 19 and 20 in \cite{DM16} hold for Dirichlet $bc$.
In the case of Neumann $bc$, by (\ref{Vneu}), (\ref{Klamda}) and by definition of $\widetilde{V}$,  the matrix representation of $(K_{\lambda}V K_{\lambda})^{Neu}$ is
\be
(K_{\lambda}V K_{\lambda})^{Neu}_{km}=A_m\left(\frac{|k-m|
\widehat{Q}^{Dir}_{|k-m|}+(k+m)\widehat{Q}^{Dir}_{k+m}}
{(\lambda-j^2)^{1/2}(\lambda-m^2)^{1/2}}\right),\nonumber
\ee
which differs from the matrix elements $(K_{\lambda}V K_{\lambda})^{Dir}_{km}$ for Dirichlet $bc$ only by the plus sign in the second term (see (5.19) in \cite{DM16}) and by the additional terms corresponding to $k$ or $m$ equals to zero. Nevertheless, in view of (\ref{HS}), inequality (\ref{ineq}) still holds when we replace $(K_{\lambda}V K_{\lambda})^{Dir}$ by $(K_{\lambda}V K_{\lambda})^{Neu}$. Hence the proofs of Lemma 19, Lemma 20, and Theorem 21 in \cite{DM16} apply to the case of Neumann $bc$
as well. Therefore we have the following Propositions:

 \begin{Proposition}
 \label{2-p-kvk} If  $v\in H^{-1}_{per} (\mathbb{R})$, there
 exist a sequence $\varepsilon_n=\varepsilon_n(v)$ decreasing to zero and $N>0$, $N\in\Gamma_{bc}$  such that
\be
 \|K_{\lambda}V K_{\lambda}\| \leq  \frac{\varepsilon_N}{2} <1 \quad
 \text{for} \quad \lambda \in H^N \backslash R_N.
 \ee
 Moreover for all $n>N$, $n\in\Gamma_{bc}$,
 \be
 \label{KVK}
 \|K_{\lambda}V K_{\lambda}\| \leq \frac{\varepsilon_n}{2}
 \ee
 for $\lambda \in H_n \backslash D_n$ if
 $bc=Per^{\pm},$ and for $\lambda \in G_n \backslash D_n$ if $bc=Dir,
 Neu$.
 \end{Proposition}

\begin{Proposition} For any potential $v\in H^{-1}_{per} (\mathbb{R})$,
the spectrum of the operator $L_{Neu}(v)$ is discrete. Moreover
there exists an integer $N$ such that
  \be
  \label{local2}
   Sp(L_{Neu}) \subset R_N \cup \bigcup_{n>N, n\in
\Gamma_{Neu}}D_n,
 \ee
and \be
 \sharp(Sp(L_{Neu})\cap R_N)=N+1, \quad
 \sharp(Sp(L_{Neu})\cap D_n)=1,\quad n>N, n\in
 \Gamma_{Neu}.
 \ee
 \end{Proposition}

\section{Main Inequalities}

For $bc=Per^\pm, Dir$ or $Neu$, we consider the Cauchy-Riesz
projections
\begin{equation}
\label{PP0} P_n = \frac{1}{2 \pi i} \int_{C_n} R_\lambda d\lambda,
\quad P^0_n = \frac{1}{2 \pi i} \int_{C_n} R^0_\lambda d\lambda,
\end{equation}
where $C_n=\partial D_n$. We estimate the norms
$\|P_n-P^0_n\|$ and
$\|D(P_n-P^0_n)\|$, where $D=\frac{d}{dx}$.
\begin{Proposition}
\label{Proposition} Let $D=\frac{d}{dx}$, $P_n$ and $P^0_n$ be
defined by (\ref{PP0}), and let $L=L_{bc}$ with $bc=Per^\pm, Dir, Neu.$
If $v\in H_{per}^{-1}([0,\pi])$  then we have, for large
enough $n,$
 \be \label{P-P0}
\|P_n-P^0_n\| \leq \varepsilon_n
\ee
and \be \label{DP-DP0}
\|D(P_n-P^0_n)\| \leq n\varepsilon_n. \ee
\end{Proposition}
\begin{proof}
In order to estimate
$\|D(P_n-P^0_n)\|$, first we note that
 \be
 \label{3-D-P}
D(P_n-P^0_n) =
\frac{1}{2\pi}\int_{C_n}D(R_\lambda-R^0_\lambda)d\lambda.
 \ee
 Indeed, using integration by parts twice one can easily see that
 \be
 \label{3-D-D-2}
\bigg{\langle} D\int_{C_n}(R_\lambda-R^0_\lambda)f d\lambda ,
g\bigg{\rangle}=\bigg{\langle} \int_{C_n}D(R_\lambda-R^0_\lambda) f
d\lambda,g \bigg{\rangle}
 \ee
  for all $f\in L^2([0,\pi])$ and $g \in C^{\infty}_0([0,\pi])$. Since
  $C^{\infty}_0([0,\pi])$ is dense in
  $L^2([0,\pi]),$ (\ref{3-D-D-2}) implies (\ref{3-D-P}).
  Hence
   \be
   \label{3-D-P-2}
\|D(P_n-P^0_n)\| \leq
\frac{1}{2\pi}\int_{C_n}
\|D(R_\lambda-R^0_\lambda)\|d|\lambda|
 \ee
$$ \quad\quad\quad\quad\quad\leq r_n\sup_{\lambda \in C_n}
\|D(R_\lambda-R^0_\lambda)\|.
$$
By (\ref{2-res}) we can write
 \be
 D (R_{\lambda}-R^0_{\lambda})= \sum_{s=1}^{\infty}D
K_{\lambda}(K_{\lambda}VK_{\lambda})^s K_{\lambda}.
 \ee
It is easy to see that
 \be \label{3-K}
\|D K_{\lambda}\|=\sup_{k \in
\Gamma_{bc}}{\frac{k}{|\lambda-k^2|^{1/2}}}=
\frac{n}{|\lambda-n^2|^{1/2}}=\frac{n}{\sqrt{r_n}},\;\;\;
 \lambda\in C_n,
 \ee
and similarly,
 \be \label{3-K2}
\|K_{\lambda}\|=\sup_{k \in
\Gamma_{bc}}{\frac{1}{|\lambda-k^2|^{1/2}}}=
\frac{1}{|\lambda-n^2|^{1/2}}=\frac{1}{\sqrt{r_n}},\;\;\;
 \lambda\in C_n,
 \ee
Note also that, since $\lambda\in C_n$, $\|K_{\lambda}VK_{\lambda}\|\leq \varepsilon_n/2 \leq 1/2$ for sufficiently large $n$'s by Proposition \ref{2-p-kvk}. Hence we obtain
 \be
 \label{3-de}
 \|D (R_{\lambda}-R^0_{\lambda})\| \leq
 \sum_{s=1}^{\infty}\|D K_{\lambda}\|
 \|K_{\lambda}VK_{\lambda}\|^s
 \|K_{\lambda}\|\quad\quad\quad\quad\quad
 \ee
$$
 \quad\quad\quad\quad\quad\quad\quad \leq 2\|D K_{\lambda}\|
 \|K_{\lambda}VK_{\lambda}\|
 \|K_{\lambda}\| \leq \frac{n \varepsilon_n}{r_n} ,\;\;\;
 \lambda\in C_n.
$$
This together with (\ref{3-D-P-2}) completes the proof of (\ref{DP-DP0}).

On the other hand, following the same argument, we see that
   \be
   \label{3-P-2}
\|P_n-P^0_n\|
\leq r_n\sup_{\lambda \in C_n}
\|R_\lambda-R^0_\lambda\|
 \ee
and
 \be
 \label{3-e}
 \|R_{\lambda}-R^0_{\lambda}\| \leq 2\| K_{\lambda}\|^2
 \|K_{\lambda}VK_{\lambda}\|
 \leq \frac{\varepsilon_n}{r_n} ,\;\;\;
 \lambda\in C_n
 \ee
which imply (\ref{P-P0}).
 \end{proof}

Let $L=L_{Per^{\pm}}$ and $L^0=L^0_{Per^{\pm}},$ and let $P_n$ and
$P_n^0$ be the corresponding projections defined by (\ref{PP0}).
Then $\mathcal{E}_n=Ran\,P_n  $  and $\mathcal{E}_n^0 =Ran\,P_n^0  $
are invariant subspaces of $L$ and $L^0,$ respectively. By Lemma 30
in \cite{DMH}, $\mathcal{E}_n$ has an orthonormal basis $\{f_n,
\varphi_n\}$ satisfying
\begin{align}
  \label{f} Lf_n & =\lambda^+_n f_n \\
 \label{phi} L\varphi_n & =\lambda^+_n \varphi_n -\gamma_n \varphi_n +\xi_n f_n.
\end{align}
We denote the quasi derivatives of $f_n$ and $\phii_n$ by $w_n$ and $u_n$, respectively,
i.e., $w_n=f_n'-Qf_n$ and $u_n=\phii'_n-Q\phii_n$. Then, in view of (\ref{iki}), we have
\be
\label{fn}
w_n'= -\lambda^+_n f_n- Q f'_n
\ee
\be
\label{phin}
u_n'= -\lambda^+_n \varphi_n- Q \varphi'_n + \gamma_n\varphi_n-\xi_n f_n.
\ee

\begin{Lemma}
\label{x+g-b+b} In the above notations, for large enough $n,$ \be
\label{9.1} \frac{1}{5}(|\beta^+_n(z_n^*)|+|\beta^-_n(z_n^*)| ) \leq
|\xi_n|+|\gamma_n| \leq 9 (|\beta^+_n(z_n^*)|+|\beta^-_n(z_n^*)| )
 \ee
\end{Lemma}

\begin{proof}  Indeed, combining (7.13) and (7.18) and (7.31)
in \cite{DMH} one can easily see that $|\xi_n| \leq
3(|\beta_n^+(z_n^*)|+|\beta_n^+(z_n^*)|)+4|\gamma_n|.$ This
inequality,  together with Lemma 20 in \cite{DMH}, implies that $
|\xi_n|+|\gamma_n| \leq 9 (|\beta_n^+(z_n^*)|+|\beta_n^+(z_n^*)|)$
for sufficiently large $n$'s. On the other hand by (7.31), (7.18),
and (7.14) in \cite{DMH} one gets  $
|\beta_n^+(z_n^*)|+|\beta_n^+(z_n^*)| \leq 5 (|\xi_n|+|\gamma_n|) $
for sufficiently large $n$'s.
\end{proof}

\begin{Proposition}
\label{mainprop}
Under the assumption $v\in H_{per}^{-1}$, there exists a sequence $\kappa_n$ converging to zero such that
for all $G \in \mathcal{E}_n$ we have
\begin{align}
\label{G-G0(0)}
|G(0)-G^0(0)| \leq & \,\kappa_n\|G\|, \quad\quad\quad \\
\label{G'-G0'}
|(G'-Q G)(0)-{G^0}'(0)| \leq &\,  n\kappa_n\|G\|,\quad\quad\quad
\end{align}
where $G^0= P_n^0 G$.
\end{Proposition}

\begin{proof} It is enough to show (\ref{G-G0(0)}) and (\ref{G'-G0'}) hold
for orthonormal basis elements $f_n$ and $\varphi_n$ in each $\mathcal{E}_n$.
We provide a proof only for $G=\varphi_n$ because the same argument proves the claim for $G=f_n$. We start with the proof of (\ref{G'-G0'}). Consider the function $\tilde{u}_n(x)=\cos mx\;u_n(x)$ where $m$ is
 an integer chosen so that $m-n$ is odd. Then $\tilde{u}_n(x)$ is satisfying $\tilde{u}_n(\pi)=- \tilde{u}_n(0)$, and therefore,
 \be
2u_n(0)=\tilde{u}_n(0)- \tilde{u}_n(\pi) =-\int_0^\pi \tilde{u}'_ndx =\int_0^\pi \left(m \sin mx\; u_n-\cos mx\;u'_n\right) dx.\nonumber
\ee
Inserting the definition of $u_n$ and the expression (\ref{phin}) for $u_n'$ into the integrand,
and applying integration by parts to the term $\int_0^\pi m\sin mx \;\phii'_n dx$ we obtain
\begin{align}
\label{akk}
2u_n(0)= & -m\int_0^\pi\sin mx \; Q \phii_n dx \\
&+\int_0^\pi \cos mx \left(Q \phii'_n +(\lambda^+_n-m^2-\gamma_n)\phii_n+\xi_n f_n\right)dx \nonumber
\end{align}
Since $\phii^0_n$ is an eigenfunction of the free operator with eigenvalue $n^2$ we also have
\be
\label{akl}
2{\phii^0_n}'(0)=(n^2-m^2)\int_0^\pi \cos mx \; \phii_n^0 dx.
\ee
Subtracting (\ref{akl}) from (\ref{akk}) we get
\be
\label{u-u}
u_n(0)- {\phii^0_n}'(0)=\frac{1}{2}\left(I_1+I_2+I_3+I_4+I_5\right),
\ee
where
\begin{align}
I_1 & = (n^2-m^2)\int_0^\pi \cos mx \left(\phii_n-\phii_n^0\right) dx,\quad
I_2 = -m\int_0^\pi \sin mx \; Q\phii_n dx, \nonumber \\
I_3 & = \int_0^\pi \cos mx \; Q\phii_n' dx,\quad
I_4 = (\lambda^+_n-n^2)\int_0^\pi\cos mx \; \phii_n dx,\nonumber \\
I_5 & = \int_0^\pi \cos mx \left(-\gamma_n \phii_n+\xi_nf_n\right)dx.\nonumber
\end{align}
Next we estimate these integrals by choosing $m$ appropriately.
By Proposition \ref{Proposition}, there is a positive sequence $\varepsilon_n$ which
dominates $\| P_n-P^0_n \|$ and converges to zero. We choose $m=m(n)$ so that $k_n= m-n$ is the largest odd number which is less than both $n$ and $1/\sqrt{\varepsilon_n}$. Then
\be
|I_1| \leq \pi k_n(2n+k_n)\|\phii_n- \phii^0_n\|_{{}_{{}_1}} \leq \frac{3\pi n}{\sqrt{\varepsilon_n}}\|\phii_n- \phii^0_n\|_{{}_{{}_2}}.
\ee
Since \be
\label{rest}
\|\phii_n- \phii^0_n\|_{{}_{{}_2}} = \|(P_n-P^0_n)\phii_n\|\leq  \|(P_n-P^0_n)\| \leq \varepsilon_n
\ee
(by Proposition \ref{Proposition}), it follows that
\be
\label{I1}
|I_1| \leq 3\pi n\sqrt{\varepsilon_n}.
\ee
In order to estimate $I_2$, we first write it as $I_2=I_{2a}+I_{2b}$ where
\be
I_{2a}=-m\int_0^\pi \sin mx \; Q \left(\phii_n-\phii^0_n\right) dx,\quad \quad I_{2b}=-m\int_0^\pi \sin mx \; Q \phii^0_n dx. \nonumber
\ee
Noting that $m=n+k_n\leq 2n$, Schwartz inequality together with (\ref{rest}) implies that
\be
\label{I2a}
|I_{2a}|\leq 2 \pi n \|Q\|_{{}_{{}_2}}\varepsilon_n.
\ee
For the second term $I_{2b}$ note that $\mathcal{E}_n^0$ is spanned by orthonormal functions $\sqrt{2}\cos nx$ and $\sqrt{2}\sin nx$, so
\be
\label{fii}
\phii_n^0=\sqrt{2}\left(a_n\cos nx +b_n\sin nx\right),
\ee
where $|a_n|^2+|b_n|^2=\|\phii_n^0\|^2=\|P^0_n\phii_n\|^2\leq \|\phii_n\|^2=1$.
Therefore, it follows that
\begin{align}
I_{2b}=-\frac{n+k_n}{\sqrt{2}}\bigg( a_n \int_0^\pi \big(\sin(2n+  k_n)x &+\sin k_nx\big)Q dx \;+ \nonumber \\
 b_n \int_0^\pi  & \big(\cos k_n x-\cos (2n+k_n)x\big) Q dx \bigg)\nonumber \\
 =  -\frac{\pi(n+k_n)}{2}\big(   a_n\big( \widehat{Q}^{Dir}_{2n+k_n}+  & \widehat{Q}^{Dir}_{k_n}\big) + b_n \big( \widehat{Q}^{Neu}_{k_n}- \widehat{Q}^{Neu}_{2n+k_n}\big)\big), \nonumber
\end{align}
Recalling $k_n \leq n$ and $|a_n|,|b_n|\leq 1$ we obtain
\be
\label{I2b}
|I_{2b}|\leq \pi n |\widehat{Q}|_n,
\ee
where we define $|\widehat{Q}|_n$ as $|\widehat{Q}|_n= |\widehat{Q}^{Dir}_{2n+k_n}|+ |\widehat{Q}^{Dir}_{k_n}| +|\widehat{Q}^{Neu}_{k_n}|+ |\widehat{Q}^{Neu}_{2n+k_n}|$. Note that $k_n$ converges to infinity by construction and $Q$ is square integrable. Hence $|\widehat{Q}|_n$ tends to zero as $n$ goes to infinity.

For $I_3$, we write it as $I_3=I_{3a}+I_{3b}$, where
\be
I_{3a}= \int_0^\pi \cos mx \; Q \left(\phii_n-\phii^0_n\right)' dx,\quad \quad I_{3b}= \int_0^\pi \cos mx \; Q {\phii^0_n}' dx. \nonumber
\ee
Applying Schwartz inequality to $I_{3a}$ we get
 \be
 \label{I3a}
 |I_{3a}|\leq \pi \|Q\|_{{}_{{}_2}}\|\left(\phii_n-\phii^0_n\right)'\|_{{}_{{}_2}} \leq \pi \|Q\|_{{}_{{}_2}}n\varepsilon_n
 \ee
since
\be
\|\left(\phii_n-\phii^0_n\right)'\|_{{}_{{}_2}} \leq \|D(P_n-P^0_n)\phii_n\|_{{}_{{}_2}}\leq \|D(P_n-P^0_n)\|\leq n\varepsilon_n
\ee
by Proposition \ref{Proposition}.

$I_{3b}$ can be treated similarly as $I_{2b}$. Inserting the derivative of (\ref{fii}) into $I_{3b}$, we obtain
\begin{align}
I_{3b}=\frac{n}{\sqrt{2}}\bigg( a_n \int_0^\pi \big(-\sin(2n+k_n)x&+\sin k_nx\big)Q dx\; + \nonumber \\
 b_n \int_0^\pi & \big(\cos k_n x+\cos (2n+k_n)x\big) Q dx \bigg) \nonumber\\
 =  \frac{\pi n}{2}\big(  a_n\big( -\widehat{Q}^{Dir}_{2n+k_n}+ & \widehat{Q}^{Dir}_{k_n}\big)+ b_n \big( \widehat{Q}^{Neu}_{k_n}+ \widehat{Q}^{Neu}_{2n+k_n}\big)\big).\nonumber
\end{align}
Hence, as for $I_{2b}$, we obtain
\be
\label{I3b}
 |I_{3b}| \leq \frac{\pi n}{2}|\widehat{Q}|_n.
 \ee
For $I_4$ we have
\be
|I_4|\leq |\lambda^+-n^2|\|\phii_n\|_{{}_{{}_1}}\leq |\lambda^+-n^2|
\ee
since $\|\phii_n\|_{{}_{{}_1}} \leq \|\phii_n\|_{{}_{{}_2}} \leq 1.$ Recalling that
each $\lambda^+_n$ lies in the disc $D_n=\{\lambda\; : \; |\lambda-n^2|< r_n\}$ where
$r_n=n\tilde{\varepsilon}_n$ we get
\be
\label{I4}
|I_4|\leq n\tilde{\varepsilon}_n.
\ee
Finally for $I_5$, in the view of Lemma \ref{x+g-b+b}, we have
\be
|I_5|\leq |\gamma_n|\|\phii_n\|+|\xi_n|\|f_n\|\leq |\gamma_n|+|\xi_n|\leq 18
(|\beta^+_n(z_n^*)|+|\beta^-_n(z_n^*)| ).\nonumber
\ee
Note that $|z_n^*|=|\frac{1}{2}(\lambda^+_n -\lambda^-_n)-n^2|$ is in the disc $D_n$ hence it is less that $n/2$ for sufficiently large $n$'s. So by Proposition 15 in \cite{DMH} there is a sequence $\hat{\varepsilon}_n$ converging to zero such that $|\beta^\pm_n(z_n^*)-V_+(\pm 2n)|\leq n \hat{\varepsilon}_n$.
Recall that $V_+(k)=ik\widehat{Q}^{Per^+}_k$. Hence
\begin{align}
|I_5|&\,\leq18\left( n \hat{\varepsilon}_n +|V_+(2n)|+|V_+(-2n)|\right)\nonumber\\
\label{I5}&\,\leq 36n \left(\hat{\varepsilon}_n +|\widehat{Q}^{Per^+}_{2n}|+|\widehat{Q}^{Per^+}_{-2n}|\right).
\end{align}
Noting that $\widehat{Q}^{Per^+}_{\pm 2n}$ converges to zero, combining (\ref{u-u}), (\ref{I1}), (\ref{I2a}), (\ref{I2b}), (\ref{I3a}), (\ref{I3b}), (\ref{I4}) and (\ref{I5})
we complete the proof of (\ref{G'-G0'}) for $G=\phii_n$.

In order to prove (\ref{G-G0(0)}) for $G=\phii_n$, now we consider the function $\hat{u}_n(x)=\sin mx\;u_n(x)$, where $m-n$ is again odd. Then
\be
0=\hat{u}_n(\pi)-\hat{u}_n(0)=\int_0^\pi \hat{u}'_ndx
=\int_0^\pi \left(m \cos mx\; u_n+\sin mx\;u'_n\right) dx.\nonumber
\ee
Substituting the definition of $u_n$ and the expression (\ref{phin}) for $u_n'$ into the integrand,
and applying integration by parts to the term $\int_0^\pi m\cos mx \;\phii'_n dx$ we obtain
\begin{align}
\label{akk2}
2 m \phii_n(0)= & -m\int_0^\pi\sin mx \; Q \phii'_n dx \\
&-\int_0^\pi \cos mx \left(Q \phii'_n +(\lambda^+_n-m^2-\gamma_n)\phii_n+\xi_n f_n\right)dx. \nonumber
\end{align}
Similarly for $\phii_n^0$ we get
\be
\label{akl2}
2m\phii^0_n(0)=-(n^2-m^2)\int_0^\pi \cos mx \; \phii_n^0 dx.
\ee
Comparing (\ref{akk2}) and (\ref{akl2}) with (\ref{akk}) and (\ref{akl}) we see that following the same argument as in the proof of (\ref{G'-G0'}) one also proves (\ref{G-G0(0)}). Note that now the multiplier $n$ disappears since $\phii_n(0)$ and $\phii_n^0(0)$ in (\ref{akk2}) and (\ref{akl2}) are also multiplied by $m$ which is greater than $n$ by our choice.
\end{proof}

\section{Proof of Theorem \ref{Asymp}}

In this section, we give a proof of Theorem \ref{Asymp}, i.e., we
show that the sequences $(|\gamma_n| +|\delta_n^{Neu}|)$  and
$(|\beta_n^- (z_n^*)|+|\beta_n^+ (z_n^*)|)$ are asymptotically
equivalent. The proof is based on the methods developed in
\cite{KaMi01,DM5,DM15}, but the technical details are different.

In the following, for simplicity, we suppress $n$ in all symbols
containing $n$. From now on, $P$  ($P^0$) denotes the Cauchy-Riesz
projection associated with $L$ ($L^0$) only. We denote the
projections associated with $L_{Neu}$  and $L^0_{Neu}$ by $P_{Neu}$
and $P^0_{Neu},$ respectively, and $\mathcal{C}=\mathcal{C}(v)$
denotes the one dimensional invariant subspace of
$L_{Neu}=L_{Neu}(v)$ corresponding to $P_{Neu}.$

\begin{Lemma} Let $f, \varphi$ be an orthonormal basis in
$\mathcal{E}$ such that (\ref{f}) and (\ref{phi}) hold. Then there is
a unit vector $G=af+b\phii$ in $\mathcal{E}$ satisfying
 \be
 \label{bcH}
  (G\hspace{0.3mm}'-QG)(0)=(G\hspace{0.3mm}'-QG)(\pi)=0,
   \ee
and there is a unit vector $g\in\mathcal{C}$ satisfying
 \be \label{dere}
 \langle G,\bar{g}\rangle \delta^{Neu}=
  b \langle \varphi,\bar{g}\rangle \gamma - b\langle f,\bar{g}\rangle \xi
\ee such that $\langle G,\bar{g}\rangle \in \mathbb{R}$ and
 \be
 \label{4950}
  \langle G,\bar{g}\rangle \geq \frac{71}{72}
  \ee
 for sufficiently large $n$.
\end{Lemma}
(Remark. (\ref{bcH}) means that $G$ is in the domain of $L_{Neu}.$)

\begin{proof} If $w(0)=0 $ then
$w(\pi)=0$ since $f$ is either a periodic or antiperiodic
eigenfunction. Hence we can set $G=f$. Otherwise we set
$\tilde{G}(x)=u(0)f(x)-w(0)\varphi(x)$. Then
$G=\tilde{G}/\|\tilde{G}\|$ satisfies (\ref{bcH}) because the
functions $f$ and $\phii$ are simultaneously periodic or
antiperiodic.

By (\ref{bcH}), $G \in Dom(L)\cap Dom (L_{Neu}), $ so we have
$L_{Neu}G=LG.$ Hence it follows
 \be \label{111111}
L_{Neu}G=aLf+bL\phii=a\lambda^+f+b(\lambda^+\phii-\gamma\ \phii+\xi
f)\ee
 $$ \quad \quad \quad \quad \quad \quad =\lambda^+(a
f+b\phii)+b(\xi f-\gamma \phii)=\lambda^+G+b(\xi f-\gamma \phii).$$

Fix a unit vector $g\in \mathcal{C}$ so that
 \be
  \label{accreditation}
   \langle G,\bar{g}\rangle= |\langle G,\bar{g}\rangle|,
 \ee
Passing to conjugates in the equation $-g^{\prime \prime} +v(x) g=
\nu g $ one can see that
 \be
 \label{de}
 L_{Neu}(\bar{v})\bar{g}=\bar{\nu}\bar{g}.
 \ee
Taking inner product of both sides of (\ref{111111}) with $\bar{g}$
we get
 \be \label{cln1} \langle L_{Neu}G,\bar{g}\rangle =
\lambda^+\langle G,\bar{g}\rangle+b(\xi\langle
f,\bar{g}\rangle-\gamma \langle \phii,\bar{g}\rangle).
 \ee
On the other hand, by (\ref{L-L}) and (\ref{de}),
we have
   \be \label{cln2} \langle L_{Neu}(v)G,\bar{g}\rangle =
   \langle G,(L_{Neu}(v))^*\bar{g}\rangle=
\langle G,L_{Neu}(\bar{v})\bar{g}\rangle
  = \nu \langle G,\bar{g}\rangle.
 \ee
  Now (\ref{cln1}) and (\ref{cln2}) imply (\ref{dere}).

Let $G^0=P^0G$ and  $\bar{g}^0=P^0_{Neu}\bar{g};$ then
$\|G^0\|,\|\bar{g}^0\| \leq 1$ since $P^0$ and $P^0_{Neu}$ are
orthogonal projections and $G$ and $\bar{g}$ are unit vectors.

We have
 $$
 \label{3IKA}
 \langle G,\bar{g}\rangle
=
 \langle G^0,\bar{g}^0\rangle +
\langle G^0,\bar{g}-\bar{g}^0\rangle + \langle G-G^0,
\bar{g}\rangle,
$$
so by the triangle and Cauchy inequalities it follows that
 $$
 \label{bodorey}
 |\langle G,\bar{g}\rangle| \geq |\langle G^0,\bar{g}^0\rangle|
 -\|\bar{g}-\bar{g}^0\|
- \|G-G^0\|.
 $$
By Proposition \ref{Proposition} we have
 \be \label{G-G0}
\|G-G^0\|=\|(P-P^0)G\| \leq \|P-P^0\| \leq \varepsilon_n
 \ee
 and similarly
  \be
  \label{g-g0}
  \|\bar{g}-\bar{g}^0\|=
  \|(P_{Neu}(\bar{v})-P^0_{Neu})\bar{g}\|
  \leq \|P_{Neu}(\bar{v})-P^0_{Neu}\| \leq
\varepsilon_n.
 \ee
  Hence, it follows that
 \be \label{oyleboyle}
  |\langle G,\bar{g}\rangle| \geq |\langle G^0,\bar{g}^0\rangle |
-2\varepsilon_n.
 \ee
Next we estimate  $|\langle G^0,\bar{g}^0\rangle |$ from below in
order to get a lower bound for $|\langle G,\bar{g}\rangle |$. Since
$\mathcal{C}^0$ is spanned by $c_n(x)= \sqrt{2}\cos nx,$ $\bar{g}^0$
is of the form
 \be \label{g0}
 \bar{g}^0 = e^{i\theta}\|\bar{g}^0\|\bigg{(}\frac{1}{\sqrt{2}}e^{inx}+
 \frac{1}{\sqrt{2}}e^{-inx}\bigg{)}
 \ee
 for some $\theta \in [0,2\pi)$. Now let $G^0_1$ and $G^0_2$
 be the coefficients of $G^0$ in the basis
 $\{e^{inx},e^{-inx}\}$, i.e.,
 \be
 \label{nelabel}
 G^0(x)=G^0_1e^{inx}+G^0_2e^{-inx}.
 \ee
 Clearly ${G^{0\hspace{0.3mm}}}'(0)=in(G^0_1-G^0_2)$. Since
 $(G\hspace{0.3mm}'-QG)=0$, by Proposition \ref{mainprop} we also have
 \be
 |{G^{0\hspace{0.3mm}}}'(0)|=|(G\hspace{0.3mm}'-QG)(0)-{G^{0\hspace{0.3mm}}}'(0)| \leq n\kappa_n.
 \ee
 Hence we obtain
 \be
 \label{ne}
 |G^0_1-G^0_2| \leq \kappa_n
 \ee
 and
 \be \label{nelabelne} |G^0_2| \leq |G^0_1|+ |G^0_1-G^0_2| \leq
 |G^0_1|+\kappa_n.
  \ee
 From (\ref{G-G0}) it follows that
  $$
 \sqrt{|G^0_1|^2+|G^0_2|^2}=\|G^0\| \geq \|G\|-\|G-G^0\| \geq
 1-\varepsilon_n,
  $$
so by (\ref{nelabelne}) we get
 \be
 \label{nelabelnelabel}
 |G^0_1| \geq \frac{1}{\sqrt{2}}-\sqrt{2}(\kappa_n+\varepsilon_n).
 \ee
On the other hand (\ref{g0}) and (\ref{nelabel})  imply
 \be
 \label{g0G0}
 |\langle G^0,\bar{g}^0 \rangle|=\frac{1}{\sqrt{2}}\|\bar{g}^0\||G^0_1+G^0_2|
 \geq \frac{1}{\sqrt{2}}\|\bar{g}^0\|\big{(}2|G^0_1|-|G^0_1-G^0_2|\big{)}.
 \ee
 Combining (\ref{ne}), (\ref{nelabelnelabel}), (\ref{g0G0})
 and taking into account that
 $$\|\bar{g}^0\| \geq \|\bar{g}\|-\|\bar{g}-\bar{g}^0\| \geq
1-\varepsilon_n$$ due to (\ref{g-g0}), we obtain
 \be
 |\langle G^0,\bar{g}^0 \rangle| \geq 1-4\varepsilon_n-2\kappa_n
 \ee
 which, together with (\ref{oyleboyle}) and (\ref{accreditation}), implies
 \be
 \langle G,\bar{g}\rangle \geq 1-6\varepsilon_n-2\kappa_n.
 \ee
Hence, for a sufficiently large $n$, $\langle G,\bar{g}\rangle \geq
71/72$.
 \end{proof}

 \begin{Corollary}
 \label{corollary}
 For sufficiently large $n$, we have
 \be
 \label{corr2} |\gamma_n|+|\delta^{Neu}_n| \leq 19
 \big{(}|\beta^+_n(z_n^*)|+|\beta^-_n(z_n^*)|\big{)}.
 \ee
 \end{Corollary}
\begin{proof}
Using (\ref{dere}), (\ref{4950}) and noting also that the absolute
values of $b$ and all inner products in the right-hand side of
(\ref{dere}) do not exceed $1$ we get $|\delta^{Neu}| \leq
72/71\big{(}|\xi|+|\gamma|\big{)}.$ This inequality, together with
Lemma \ref{x+g-b+b}, implies (\ref{corr2}).
 \end{proof}

Corollary \ref{corollary} proves the second inequality in
(\ref{A10}). In order to complete the proof of Theorem~\ref{Asymp} it
remains to prove the first inequality in (\ref{A10}).

By Proposition 34 in \cite{DMH}, if
 \be \label{case1}
  \textit{Case 1}: \quad \quad \quad \frac{1}{4}|\beta^-(z^+)| \leq
 |\beta^+(z^+)| \leq 4 |\beta^-(z^+)|,\quad\quad\quad\quad\quad
 \ee
 then we have
 \be
 \label{case11}
 |\beta^+(z^*)|+|\beta^-(z^*)| \leq 2|\gamma|.\quad \quad
 \ee
 Next we consider the complementary cases
$$
 \textit{Case 2(a)}:\quad 4|\beta^+(z^+)| < |\beta^-(z^+)|\quad
 \text{or} \quad \textit{Case 2(b)}:\quad 4|\beta^-(z^+)| <
 |\beta^+(z^+)|.$$
\begin{Lemma}
\label{LLL} If {\em Case 2(a)} or {\em Case 2(b)} holds, then we
have, for sufficiently large $n,$
 \be \label{LLLL}
\frac{1}{4} \leq \frac{|w(0)|}{|u(0)|} \leq 4.
 \ee
 \end{Lemma}

\begin{proof} We consider only  \textit{Case 2(a)}, since the proof \textit{Case 2(b)} is
 similar. Let $f^0=P^0f$, $\phii^0=P^0\phii$ and let
$ f^0=f^0_1 e^{inx}+f^0_2 e^{-inx} $ and $ \phii^0=\phii^0_1
e^{inx}+\phii^0_2 e^{-inx}. $
 In \textit{Case 2(a)}, if $v \in L^2([0,\pi])$
 it was shown in the proof of Lemma 64 in \cite{DM15} that the following
 inequalities hold (inequalities (4.51), (4.52), (4.54), and (4.55) in \cite{DM15}):
 \be
 \label{6-f}
 |f^0_1| \geq \frac{2}{\sqrt{5}}-2\rho_n,
 \quad|f^0_2| \leq \frac{1}{\sqrt{5}},\quad |\phii^0_1| \leq \frac{1}{\sqrt{5}}+\rho_n,
 \quad|\phii^0_2|\geq
 \frac{2}{\sqrt{5}}-2\rho_n,
 \ee
where $\rho_n$ is a sequence converging to zero. These
inequalities were derived using Lemma 21 and Proposition 11 in
\cite{DM15} which still hold in the case where $v \in
H^{-1}([0,\pi])$, (see Lemma
6 and Proposition 44 in \cite{DMH}) \footnote{In the derivation of the inequalities (\ref{6-f}), Proposition 44 in \cite{DMH} is needed for its corollary (\ref{P-P0}). So one can directly use (\ref{P-P0}) instead of Proposition 44 in \cite{DMH} to show (\ref{6-f}) hold.}. Hence we can safely use them.

Note that ${f^{0\hspace{0.3mm}}}'(0)=in(f^0_1-f^0_2)$. Using (\ref{6-f}) we get
 \be
 |{f^{0\hspace{0.3mm}}}'(0)| \geq n(|f^0_1|-|f^0_2|) \geq
 n\bigg{(}\frac{1}{\sqrt{5}}-2\rho_n\bigg{)} \geq
 \frac{n}{\sqrt{6}}
 \ee
for sufficiently large $n$. On the other hand we have
 \be
 |{f^{0\hspace{0.3mm}}}'(0)| \leq n(|f^0_1|+|f^0_2|) \leq n\sqrt{2}\|f^0\| \leq
 \sqrt{2}n
 \ee
Following the same argument for ${\phii^{0\hspace{0.3mm}}}'(0)$, we have both
 \be
  \frac{n}{\sqrt{6}} \leq |{f^{0\hspace{0.3mm}}}'(0)| \leq \sqrt{2}n\quad\text{and}\quad
  \frac{n}{\sqrt{6}} \leq |{\phii^{0\hspace{0.3mm}}}'(0)| \leq \sqrt{2}n.
 \ee
On the other hand by Proposition \ref{mainprop} we have
 \be
 |w(0)-{f^{0\hspace{0.3mm}}}'(0)| \leq n\kappa_n \quad \text{and} \quad |u(0)-{\phii^{0\hspace{0.3mm}}}'(0)| \leq n\kappa_n.
 \ee
 Hence, for sufficiently large $n$'s, we get
 \be
 \frac{|w(0)|}{|u(0)|} \leq
 \frac{|{f^{0\hspace{0.3mm}}}'(0)|+n\kappa_n}{|{\phii^{0\hspace{0.3mm}}}'(0)|-n\kappa_n}\leq
 \frac{n(\sqrt{2}+\kappa_n)}{n(1/\sqrt{6}-\kappa_n)}\leq 4
 \ee
 and
 \be
 \frac{|w(0)|}{|u(0)|} \geq
 \frac{|{f^{0\hspace{0.3mm}}}'(0)|-n\kappa_n}{|{\phii^{0\hspace{0.3mm}}}'(0)|+n\kappa_n}\geq
 \frac{n(1/\sqrt{6}-\kappa_n)}{n(\sqrt{2}+\kappa_n)}\geq \frac{1}{4}.
 \ee
 \end{proof}

\begin{Proposition}
 \label{prop2}
For sufficiently large $n$, we have
 \be
 \label{corr} \big{(}|\beta^+_n(z_n^*)|+|\beta^-_n(z_n^*)|\big{)} \leq
 80 (|\gamma_n|+|\delta^{Neu}_n|)
 \ee
 \end{Proposition}

\begin{proof} In view of
 (\ref{case11}), it remains to prove (\ref{corr}) if
 \textit{Case 2(a)} or \textit{Case 2(b)} holds.

 Now (\ref{dere}) implies that
 \be \label{yokoyleyok}
 |b||\langle f,\bar{g}\rangle | |\xi| \leq |
 \delta^{Neu}|+|\gamma|.
 \ee
 Thus, in order to estimate $|\xi|$ from above by
 $|\delta^{Neu}|+|\gamma|$ we need to
 find a lower bound to $|b||\langle f,\bar{g}\rangle |$. We have
 \be
 \label{aydiaydi}
 |b||\langle f,\bar{g}\rangle |=|b|\big{|}\langle f,G \rangle+\langle f,\bar{g}-G
 \rangle\big{|}\geq |a||b|-\|\bar{g}-G\|
 \ee
since $\|f\|=1$, $|b| \leq 1$ and $\langle f,G \rangle= \bar{a}.$ In
view of (\ref{4950})
 \be
\|\bar{g}-G\|^2=\|\bar{g}\|^2+\|G\|^2-2 Re \,
\langle\bar{g},G\rangle=2-2\langle\bar{g},G\rangle \leq
\frac{1}{36},
 \ee
  hence
 \be \label{uwl} \|\bar{g}-G\| \leq\frac{1}{6}. \ee
On the other hand, by the construction of $G$ we know $|b/a|=
|w(0)/u(0)|,$ so Lemma \ref{LLL} implies that $1/4\leq |b/a|
\leq 4.$ Since $|a|^2+|b|^2=1$, a standard calculus argument shows
that
 \be \label{sar1}
 |a||b| \geq\frac{4}{17}.
 \ee
In view of (\ref{uwl}) and (\ref{sar1}), the right-hand side of
(\ref{aydiaydi}) is not less than $4/17-1/6 > 1/15,$ i.e.,
$|b||\langle f,\bar{g}\rangle |>1/15.$ Hence, by (\ref{yokoyleyok}),
it follows that
 \be \label{bsbsb}
 |\xi| \leq 15 \big{(} |\delta^{Neu}|+|\gamma| \big{)}.
 \ee
Now we complete the proof combining (\ref{bsbsb}) and Lemma
\ref{x+g-b+b}. \end{proof}

Corollary \ref{corollary} and Proposition \ref{prop2} show that
(\ref{A10}) holds, so Theorem~\ref{Asymp} is proved.

\section*{Acknowledgement}
The author wishes to express his gratitude to Prof. Plamen Djakov
who suggested the problem and offered invaluable assistance, support
and guidance.

\end{document}